\documentclass[letterpaper, 10 pt, conference]{ieeeconf}  

\usepackage{amsmath}
\usepackage{amssymb}
\usepackage{comment}
\usepackage{mathtools}

\usepackage{xcolor}
\usepackage{circuitikz}
\usepackage{tikz}
\usepackage{hyperref}
\usepackage[capitalize]{cleveref}
\usepackage{subcaption}
\usepackage{algorithm}
\usepackage{algorithmic}
\usetikzlibrary{arrows.meta}
\usetikzlibrary{positioning}
\usetikzlibrary{calc}
\mathtoolsset{showonlyrefs}
\hypersetup{
    colorlinks=true,
    linkcolor=black,
    citecolor=black,
    urlcolor=black
}

\newcommand{\Mset}{\langle M\rangle}
\renewcommand{\Re}{\mathbb{R}}
\newcommand{\Ne}{\mathbb{N}}

\newtheorem{theorem}{Theorem}
\newtheorem{lemma}{Lemma}
\newtheorem{proposition}{Proposition}
\newtheorem{definition}{Definition}
\newtheorem{remark}{Remark}

\newtheorem{problem}{Problem}
\newif\ifextended
\extendedtrue 
\IEEEoverridecommandlockouts                              
\title{\LARGE \bf
Iterative graph lifting for automatic design of path-complete stability certificates
}

\author{Léa Ninite and Raphaël M.~Jungers
\thanks{L.~Ninite is a Research Fellow of the Fonds de la Recherche Scientifique~–~FNRS. R.~M.~Jungers is a FNRS honorary Research Associate. This project has
received funding from the European Research Council (ERC) under the
European Union’s Horizon 2020 research and innovation programme under
grant agreement No 864017 - L2C, from the Horizon Europe programme
under grant agreement No101177842 - Unimaas, and from the ARC (French
Community of Belgium)~-~project name: SIDDARTA. L.~Ninite and R.~M.~Jungers are with the
ICTEAM, UCLouvain (Louvain-la-Neuve, Belgium). E-mail addresses: {\tt\small \{lea.ninite, raphael.jungers\}@uclouvain.be}}%
}

\begin{document}

\maketitle
\thispagestyle{empty}
\pagestyle{empty}

\begin{abstract}
Stability of switched linear systems under arbitrary switching is a fundamental problem in control theory, closely related to the joint spectral radius (JSR), which characterizes the worst-case growth rate of system trajectories.
In this paper, we contribute to the path-complete approach for approximating the JSR. This framework constructs algebraic stability certificates using labeled directed graphs, known as path-complete graphs. These certificates can be computed via an optimization problem. We propose an iterative algorithm that refines path-complete graphs in an efficient and parsimonious manner. The algorithm relies on a graph-theoretic analysis of the optimality conditions of the underlying optimization problem. In particular, we derive a sufficient condition under which the exact JSR is attained by a given path-complete graph. When this condition is not satisfied, we identify bottleneck nodes by analyzing the graph induced by the active constraints. We then use this information to refine the path-complete graph via local graph lifting (node splitting), and repeat the procedure. Numerical experiments demonstrate the effectiveness and scalability of the proposed approach, outperforming state-of-the-art methods on all challenging instances tested.
\end{abstract}

\section{INTRODUCTION}
Switched systems arise naturally in many areas of control and dynamical
systems~\cite{liberzon2003switching,jungers2009joint}. They appear in a
variety of applications including networked control systems~\cite{donkers2011stability}, mechanical and engineering systems~\cite{blanchini2012constant}, and biological or epidemiological models~\cite{hernandez2011optimal}.

A fundamental problem in the analysis of such systems is the assessment
of their stability under \emph{arbitrary} switching. For discrete-time switched
linear systems, this question is closely related to the
\emph{joint spectral radius} (JSR), introduced by Rota and Strang~\cite{rota1960jsr}. The JSR captures the worst-case growth of products
of matrices taken from a finite set and characterizes stability under arbitrary switching. However, computing or even
approximating the JSR is known to be computationally challenging in
general~\cite{blondel2000boundedness,tsitsiklis1997lyapunov}; see also~\cite{jungers2009joint} for a survey of the theory and applications of
the JSR.

Classical methods for certifying stability construct a \emph{common Lyapunov function} within a prescribed function class (e.g., quadratic functions that are LMI-representable~\cite{boyd1994lmi}) by solving an optimization problem. However, these methods suffer from conservatism induced by this restriction, as a valid common Lyapunov function may require a richer representation than the prescribed class.

To reduce this conservatism, richer classes of Lyapunov certificates have been proposed, including polynomial Lyapunov functions computed via \emph{sum-of-squares} (SOS) techniques~\cite{parrilo2008approximation} and approaches based on \emph{multiple Lyapunov functions}~\cite{branicky1998multiple, lee2006uniform}.

More recently, the \emph{path-complete Lyapunov framework} introduced in~\cite{ahmadi2014joint} appeared as a unifying perspective. It associates Lyapunov functions with the nodes of a
labeled directed graph and enforces decrease conditions along its
edges. If the graph satisfies a \emph{path-completeness} property (see
\Cref{def:PC} below), these inequalities guarantee that the resulting
optimization problem provides an upper bound on the JSR.

From early numerical experiments, it quickly became clear that the
choice of the path-complete graph is crucial for the quality of the
resulting bound on the JSR. Empirically, larger graphs (in terms
of their number of nodes or edges) tend to provide tighter bounds.
However, constructing suitable graphs for a given system remains a
challenging task, and only a few systematic families of path-complete
graphs are currently known.

A notable systematic construction arises from Lyapunov functions
that depend on a finite history of the switching signal. This idea
appeared in early work on switch-dependent quadratic Lyapunov
functions~\cite{bliman2003stability}. In the path-complete Lyapunov framework, these
constructions correspond to \emph{De Bruijn graphs}~\cite{de1948combinatorial}. Increasing the graph order yields tighter JSR bounds that converge asymptotically to the true value~\cite[Theorem~6.1]{ahmadi2014joint}. However, the
size of these graphs grows exponentially with the order, leading to
large optimization problems.

This motivates the search for systematic procedures to refine path-complete graphs and obtain tighter JSR bounds without the exponential graph growth of memory-based constructions such as De Bruijn graphs.

To this end, we propose a theoretical analysis of the link
between the optimality conditions of the optimization problem
and the topological structure of the graph. This analysis leads to a systematic
procedure to refine path-complete graphs in a parsimonious and efficient
manner. 
\paragraph*{Contributions} Our contributions are threefold.

First, we introduce a local graph transformation, referred to as a \emph{lifting} operation. Given a node $v$, the lift splits $v$ into multiple copies, one for each of its out-neighbors, thereby refining the graph locally without uniformly increasing the memory of the switching signal as in De Bruijn constructions. This transformation is shown to preserve path-completeness, a particularly desirable property since verifying path-completeness of a labeled graph is PSPACE-complete~\cite[Theorem~4]{jungers2017characterization}.

Second, we analyze the structure of the subgraph induced by the tight
constraints of the path-complete Lyapunov
optimization problem. We show that this structure provides a simple
certificate to determine whether a given path-complete graph already
yields the exact joint spectral radius.

Finally, we exploit this insight to design an iterative algorithm that
refines path-complete graphs through local lifting operations. The
algorithm uses the structure of the tight constraints to identify nodes
whose lifting may improve the bound on the JSR.
\paragraph*{Related work}
Graph lifting operations were considered in the context of
polyhedral path-complete Lyapunov functions~\cite{athanasopoulos2019polyhedral}, where several candidate lifts are
generated and the one yielding the best JSR bound is selected after
solving the corresponding optimization problems. Our lifting operation
differs from that of~\cite{athanasopoulos2019polyhedral}, and is coupled
with an analysis of the optimization problem itself: rather than
selecting among candidate lifts, we exploit the structure induced by
the tight constraints of the optimal solution to determine which lifting operation
may improve the bound.

More recently,~\cite{debauche2026} proposed an SMT-based method to construct path-complete stability certificates. While this approach produces compact path-complete graphs, it suffers from the limited scalability of SMT solvers. In contrast, we combine optimization theory and graph-theoretic analysis to iteratively build efficient and scalable stability certificates.
\paragraph*{Outline}
The remainder of the paper is organized as follows.
Section~\ref{sec:preliminaries} reviews the necessary background.
Section~\ref{sec:lifting} introduces the proposed lifting operation and
demonstrates its main properties.
Section~\ref{sec:tight} analyzes the structure of the tight subgraph
(see \Cref{def:tight_subgraph}) and establishes the optimality
certificate.
Section~\ref{sec:algorithm} presents the lifting algorithm.
Finally, Section~\ref{sec:experiments} illustrates the effectiveness of
the proposed approach and its improved scalability compared to De Bruijn constructions.

\paragraph*{Notations}
Given $M\in \mathbb N$, we let $\Mset\coloneqq \{1,\dots,M\}$. 
For a set of matrices $\mathcal A=\{A_1,\dots,A_M\}\subseteq \Re^{n\times n}$, 
$\rho(\mathcal A)$ denotes the joint spectral radius\ifextended, while for a matrix 
$A \in \Re^{n\times n}$, $\rho(A)$ denotes the spectral radius.
\fi
\section{PRELIMINARIES}\label{sec:preliminaries}
\subsection{Problem: stability of arbitrarily switching systems}

Consider a finite set of matrices $\mathcal A = \{A_1,\dots,A_M\} \subset \mathbb{R}^{n\times n}$.
A \emph{switched linear system} is defined as
\begin{equation}
x_{k+1} = A_{\sigma(k)} x_k, \label{eq:switched_syst}
\end{equation}
where $x_k\in\mathbb{R}^n$ is the \emph{state} at time $k$, and
$\sigma(k)\in\Mset$ is the \emph{mode} at time $k$.
The function $\sigma : \mathbb{N} \to \Mset$ is called the
\emph{switching signal}.

We consider \emph{arbitrary switching}, meaning that the switching signal is viewed as an external \emph{uncontrolled} signal.

The stability of the arbitrarily switching system \eqref{eq:switched_syst}
is characterized by the \emph{joint spectral radius} (JSR) of the matrix
set $\mathcal A$. The joint spectral radius of $\mathcal A$ is defined as
\begin{equation}
\rho(\mathcal A) \coloneqq
\lim_{k\to\infty}
\max_{\sigma\in\Mset^k}
\|A_{\sigma_k} \cdots A_{\sigma_2} A_{\sigma_1}\|^{1/k},
\end{equation}
and is independent of the chosen norm.

The JSR corresponds to the worst-case asymptotic growth
rate of all products of matrices in $\mathcal A$. In particular, the
switched system \eqref{eq:switched_syst} is stable
if and only if $\rho(\mathcal A) < 1$ (see e.g.~\cite[Corollary 1.1]{jungers2009joint}).

Computing the JSR is
notoriously difficult and determining whether $\rho(\mathcal A)\le 1$
is an undecidable problem~\cite[Theorem 2]{blondel2000boundedness}. As a consequence,
most existing approaches aim at computing tractable \emph{bounds}
on $\rho(\mathcal A)$.

\subsection{The path-complete framework}
Among the existing approaches to approximate the JSR, the \emph{path-complete Lyapunov framework}, introduced in~\cite{ahmadi2014joint},
provides a systematic way to derive tractable stability conditions and
upper bounds on the JSR using directed labeled graphs and multiple
Lyapunov functions.

A directed labeled
graph on $\Mset$ is a pair $G=(S,E)$, where $S$ is a finite set
of nodes and $E \subseteq S \times S \times \Mset$ is a set of
labeled edges. An element $(a,b,i)\in E$ represents a directed edge from
$a$ to $b$ labeled by $i$, corresponding to a transition under mode $i$.

Intuitively, a directed labeled graph is path-complete if every possible switching
sequence can be realized as a path in the graph. This is formalized below.

\begin{definition}[Path-complete graph] \label{def:PC}
Let $G = (S,E)$ be a directed labeled graph on $\Mset$.
We say that $G$ is \emph{path-complete} if for any $K\in\mathbb{N}_{>0}$
and any sequence $\sigma = (i_1,\ldots,i_K)\in\Mset^K$, there
exists a path 
$\{(a_k,a_{k+1}, i_k)\}_{k=1}^K\subseteq E$.
\end{definition}

\Cref{fig:pc-graph} shows a path-complete graph. The graph
in~\Cref{fig:nonpc-graph} is not path-complete; for instance,
it cannot generate the sequence $22$.
\begin{figure}[ht] 
\centering 
\subfloat[]{%
\begin{tikzpicture}[>=Stealth, scale=1, transform shape, node distance=2.3cm, vertex/.style={circle, draw, minimum size=0.7cm, font=\small, fill=gray!20}] \node[vertex] (V1) {$a_1$}; \node[vertex, right of=V1] (V2) {$a_2$}; \path[->] (V1) edge[loop above] node{$1$} (V1); \path[->] (V2) edge[loop above] node{$2$} (V2); \path[->] (V1) edge[bend left=20] node[above]{$1$} (V2); \path[->] (V2) edge[bend left=20] node[below]{$2$} (V1); \end{tikzpicture} \label{fig:pc-graph} } \hspace{1.3cm} \subfloat[]{%
\begin{tikzpicture}[>=Stealth, scale=1, transform shape, node distance=2.3cm, vertex/.style={circle, draw, minimum size=0.7cm, font=\small, fill=gray!20}] \node[vertex] (V1) {$a_1$}; \node[vertex, right of=V1] (V2) {$a_2$}; \path[->] (V1) edge[loop above] node{$1$} (U1); \path[->] (V1) edge[bend left=20] node[above]{$1$} (V2); \path[->] (V2) edge[bend left=20] node[below]{$2$} (V1); \end{tikzpicture} \label{fig:nonpc-graph}} \caption{(a) Path-complete graph with two nodes, for a system with two switching modes. (b) Graph not path-complete.} \label{fig:pc-vs-copc} \end{figure}
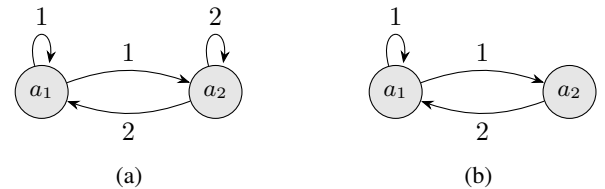

Path-complete graphs are used to certify stability of arbitrarily switching systems by associating a Lyapunov function to each node and enforcing decrease conditions along the edges of the graph. We formalize this
idea below.
\smallskip
\begin{definition}[Path-complete Lyapunov function]
Given $\mathcal A=\{A_1,\dots,A_M\}\subset\mathbb{R}^{n\times n}$, a
\emph{path-complete Lyapunov function} (PCLF) for $\mathcal A$ is a pair
$(G,V_S)$, where $G=(S,E)$ is a path-complete graph on $\Mset$ and
$V_S\coloneqq\{V_a:\mathbb{R}^n\to\mathbb{R}_{\ge0}\}_{a\in S}$ is a set of continuous, positive definite,
and radially unbounded functions satisfying
\[
V_b(A_i x) \le V_a(x),
\qquad \forall (a,b,i)\in E,\ \forall x\in\mathbb{R}^n.
\]
\end{definition}

\smallskip

The existence of a PCLF provides bounds on the joint spectral radius, as formalized in the following theorem.
\begin{theorem}[{\cite[Theorem~2.4]{ahmadi2014joint}}]
\phantomsection
\label{thm:ub_jsr}
Let $\mathcal A=\{A_1,\dots,A_M\}$ and
$G$ be a path-complete graph. For $\gamma>0$, set
$\mathcal A_\gamma \coloneqq \{\gamma^{-1}A_i : i\in\Mset\}$.
If $\mathcal A_\gamma$ admits a PCLF on $G$, then
$\rho(\mathcal A)\le \gamma$.
\end{theorem}

In practice, the functions $V_a$ are often restricted to a tractable
class. For quadratic Lyapunov functions, the smallest $\gamma$
satisfying the Lyapunov inequalities associated with $G$ is obtained by solving the following optimization problem.

\begin{problem}
\phantomsection
Let $G=(S,E)$ be a directed labeled graph on $\Mset$ and let
$\mathcal A=\{A_1,\dots,A_M\}$ be a set of matrices. We consider the
following optimization problem:
\begin{align*}
\begin{aligned}
\inf_{\{P_a\}_{a\in S}, \gamma\ge 0} \quad & \gamma \\[1mm]
\text{s.t.} \quad
& \gamma^2 P_a \succeq A_i^\top P_b A_i, && \forall (a,b,i)\in E, \\[1mm]
& P_a \succ 0, && \forall a\in S .
\end{aligned}
\end{align*}
\label{prob:graph_opt}
\end{problem}

We denote by $\gamma^\star(G)$ the optimal value of
\Cref{prob:graph_opt} on~$G$. 
\begin{remark}
The theoretical results in this paper rely on a subgraph
$\tilde G=(S,\tilde E)$ satisfying
$\gamma^\star(\tilde G)=\gamma^\star(G)$. We refer to such a set
$\tilde E\subseteq E$ as a \emph{support set}. In practice, one easily obtains a support set from a numerical solution by identifying the tight edges (see Definition~\ref{def:tight_subgraph} below), although such a procedure may fail when the optimum is not attained. In this work, we make the assumption that it does, for the sake of simplicity. The theoretical results could, however, be extended to any support set.
\end{remark}

If $G$ is path-complete, then
$
\rho(\mathcal A) \le \gamma^\star(G)
$. The quality of this upper bound $\gamma^\star(G)$ strongly depends on the
choice of the graph $G$. To construct path-complete graphs that yield tighter bounds,
we introduce in the next section a local lift that lies at the core of our algorithm.

\section{THE FORWARD LIFT}\label{sec:lifting}
The forward lift of a graph $G$ with respect to a node $v$ splits $v$ into one copy
for each out-neighbor of $v$.\footnote{One can analogously define a dual notion, the \emph{backward lift}, by considering in-neighbors instead of out-neighbors.}

More formally, letting
\[
\mathrm{Post}(v) \coloneqq
\{ w \in S : \exists i,\ (v,w,i)\in E \}
\]
denote the set of out-neighbors of $v$, the forward lift is defined as follows.
\begin{definition}[Forward lift]
Let $G=(S,E)$ be a directed labeled graph on $\Mset$ and consider $v \in S$. The \emph{forward lift} of $G$ with respect to $v$ is the directed graph $G^{v}=(S^{v},E^{v})$, where the node set is defined by
\[
S^{v} \coloneqq (S \setminus \{v\}) \;\cup\; \{ v^{(w)} : w \in \mathrm{Post}(v) \}
\]
and the edge set is defined by
\begin{equation}\label{eq:Ev}
E^{v} \coloneqq E_1 \cup E_2 \cup E_3 \cup E_4,
\end{equation}
where
\begin{align*}
E_1 &\coloneqq \{ (a,b,i) \in E : a\neq v, b\neq v \}, \\
E_2 &\coloneqq \{ (a, v^{(w)}, i) : (a,v,i) \in E, \ w \in \mathrm{Post}(v), \ a \ne v\}, \\
E_3 &\coloneqq \{ (v^{(w)}, w, i) : (v,w,i) \in E, \ w\ne v \}, \\
E_4 &\coloneqq \{ (v^{(v)}, v^{(w)}, i) : (v,v,i) \in E, \ w \in \mathrm{Post}(v) \}.
\end{align*}
\vspace{0.05em}
\end{definition}

An example of forward lift is shown in~\Cref{fig:forward-lift}.
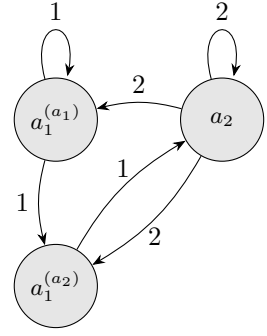
\begin{figure}[H]
\centering
\begin{tikzpicture}[>=Stealth, scale=1, transform shape,
    node distance=2.2cm,
      vertex/.style={
        circle,
        draw,
        fill=gray!20,
        minimum width=1.1cm,
        minimum height=1.1cm,
        inner sep=0pt,
        font=\small
    }]

\node[vertex] (V11) {$a_1^{(a_1)}$};
\node[vertex, right of=V11] (V2) {$a_2$};
\node[vertex, below of=V11] (V12) {$a_1^{(a_2)}$};

\path[->] (V11) edge[loop above] node{$1$} (V11);
\path[->] (V2) edge[loop above] node{$2$} (V2);

\path[->] (V11) edge[bend right=15] node[left]{$1$} (V12);
\path[->] (V12) edge[bend left=15] node[above]{$1$} (V2);

\path[->] (V2) edge[bend right=15] node[above]{$2$} (V11);
\path[->] (V2) edge[bend left=15] node[below]{$2$} (V12);

\end{tikzpicture}
\caption{Forward lift of the graph in \Cref{fig:pc-graph} with respect to node $a_1$.}
\label{fig:forward-lift}
\end{figure}
In the theorem below, we prove that the forward lift preserves path-completeness. 
\begin{theorem}
Let $G=(S,E)$ be a path-complete graph. Then, for any node $v\in S$, the forward lift of $G$ with respect to $v$ is path-complete.
\end{theorem}
\ifextended
\begin{proof}
Let $K\in \mathbb N_{>0}$ and let $\sigma=(i_1,\ldots,i_K)\in \Mset^K$ be any sequence of labels. Since $G$ is path-complete, there exists a path $\{(a_k,a_{k+1},i_k)\}_{k=1}^K \subseteq E$.

Let $v\in S$ and let $G^v=(S^v,E^v)$ be the forward lift of $G$ with respect to $v$. We construct a path in $G^v$ with the sequence of labels $\sigma$. Define $b_1,\ldots,b_{K+1}\in S^v$ as follows. 
If $a_k\neq v$, set $b_k=a_k$. If $a_k=v$ and $k\le K$, set $b_k=v^{(a_{k+1})}$. 
If $a_{K+1}=v$, choose any $w\in\mathrm{Post}(v)$ and set $b_{K+1}=v^{(w)}$.

We verify that $(b_k,b_{k+1},i_k)\in E^v$ (defined in \eqref{eq:Ev}) for all $k=1,\ldots,K$:
\begin{itemize}
\item if $a_k\neq v$ and $a_{k+1}\neq v$, then $(b_k,b_{k+1},i_k)\in E_1$,
\item if $a_k\neq v$ and $a_{k+1}=v$, then $(b_k,b_{k+1},i_k)\in E_2$,
\item if $a_k=v$ and $a_{k+1}\neq v$, then $(b_k,b_{k+1},i_k)\in E_3$,
\item if $a_k=v$ and $a_{k+1}=v$, then $(b_k,b_{k+1},i_k)\in E_4$.
\end{itemize}

Therefore, $\{(b_k,b_{k+1},i_k)\}_{k=1}^K \subseteq E^v$, which shows that $G^v$ is path-complete.
\end{proof}
\fi

The following theorem shows that solving \Cref{prob:graph_opt} on a forward lift of a graph (not necessarily path-complete) yields a \emph{relaxation} of the problem on the original graph.
\begin{theorem}
\phantomsection
\label{thm:relaxation_forward}
Let $G=(S,E)$ be a directed labeled graph and let $v\in S$.
Let $G^v$ be the forward lift of $G$ with respect to $v$.
Then $\gamma^\star (G^v) \le \gamma^\star (G)$.
\end{theorem}
\ifextended
\begin{proof}
Let $(\gamma,\{P_a\}_{a\in S})$ be feasible for~\Cref{prob:graph_opt}
on $G=(S,E)$. We construct a feasible solution for the forward lift $G^v=(S^v,E^v)$ with the same value $\gamma$.

We define matrices on $S^v$ as follows:
\[
\begin{aligned}
\tilde P_a &\coloneqq P_a && \text{for } a\in S\setminus\{v\}, \\
\tilde P_{v^{(w)}} &\coloneqq P_v && \text{for all } w\in\mathrm{Post}(v).
\end{aligned}
\]

We verify feasibility for each type of edge in $E^v$ as defined in \eqref{eq:Ev}.

\paragraph{Edges in $E_1$}
These are edges $(a,b,i)\in E$ with $a\neq v$, $b\neq v$.
The inequality
\[
\gamma^2 \tilde P_a = \gamma^2 P_a \succeq A_i^\top P_b A_i =A_i^\top \tilde P_b A_i
\]
holds by feasibility for $G$.

\paragraph{Edges in $E_2$}
These are edges $(a,v^{(w)},i)$ with $(a,v,i)\in E$ and $a\neq v$.
The required inequality is
\[
\gamma^2 \tilde P_a = \gamma^2 P_a \succeq A_i^\top P_v A_i
=A_i^\top \tilde P_{v^{(w)}} A_i
,
\]
which holds because $(a,v,i)\in E$ in $G$.

\paragraph{Edges in $E_3$}
These are edges $(v^{(w)},w,i)$ with $(v,w,i)\in E$ and $w\neq v$.
We must check
\[
\gamma^2 \tilde P_{v^{(w)}} \succeq A_i^\top P_w A_i,
\]
which holds by feasibility for $G$ since $\tilde P_{v^{(w)}}=P_v$.

\paragraph{Edges in $E_4$}
These are edges $(v^{(v)}, v^{(w)},i)$ with $(v,v,i)\in E$. The inequality 
\[
\gamma^2 \tilde P_{v^{(v)}} = \gamma^2 P_v \succeq A_i^\top P_v A_i = A_i^\top \tilde P_{v^{(w)}} A_i
\]
is satisfied by feasibility for $G$.

Thus $(\gamma,\{\tilde P_a\})$ is feasible for $G^v$.
Therefore every feasible solution for $G$ yields a feasible solution for $G^v$ with the same objective value, which concludes the proof.
\end{proof}
\fi

\section{THE TIGHT SUBGRAPH}\label{sec:tight}
In this section, we analyze the structure of the \emph{tight subgraph}, defined by the constraints that are active at a given solution of \Cref{prob:graph_opt}. This analysis reveals when the current path-complete graph already provides a tight bound on the JSR and when additional forward lifts may still improve the bound, which will guide the design of the algorithm in \Cref{sec:algorithm}.
\subsection{Definition and useful property}

The definition of \emph{tight subgraph} is formalized below.

\begin{definition}[Tight subgraph]\label{def:tight_subgraph}
Given a directed labeled graph $G=(S,E)$, let $(\gamma,\{P_a\}_{a\in S})$ be a feasible solution of \Cref{prob:graph_opt}.  
The \emph{tight subgraph} associated with this solution is the directed labeled graph $\tilde G=(S,\tilde E)$, where $\tilde E \subseteq E$ is defined as
\begin{equation*}
\tilde E
\coloneqq
\left\{
(a,b,i)\in E \;\middle|\;
\lambda_{\min}\!\left(\gamma^2 P_a - A_i^\top P_b A_i\right)=0
\right\}.
\end{equation*}
\smallskip
\end{definition}

In the next proposition, we show that solving \Cref{prob:graph_opt} on the original graph or on the tight subgraph yields the same upper bound on the JSR.\footnote{In other words, the set of tight constraints forms a
\emph{support set}.} While this property is well known for convex programs, note that \Cref{prob:graph_opt} is non-convex. However, for any fixed $\gamma \ge 0$, it reduces to a convex semidefinite feasibility problem.
\begin{proposition}
\phantomsection
\label{prop:opt_tight}
Let $G=(S,E)$ be a directed labeled graph, and suppose that \Cref{prob:graph_opt} admits an optimal solution $(\gamma^\star ,\{P_a^\star \}_{a\in S})$ on $G$. Let $\tilde G=(S,\tilde E)$ be the tight subgraph associated with this solution.
Then $\gamma^\star (\tilde G) = \gamma^\star (G)$.
\end{proposition}
\ifextended
\begin{proof}
The proof is given in Appendix~\ref{appendix:proof_equality_tight}.
\end{proof}
\fi
\subsection{Cycle-based lower bound on the JSR}
In this subsection, we consider graphs that are \emph{not} path-complete, in particular directed cycles, for which \Cref{prob:graph_opt} yields \emph{lower} bounds on the JSR (see \Cref{lemma:lb_on_cycle}). While these results do not involve the tight subgraph directly, they serve as auxiliary technical tools for deriving the graph-theoretic optimality certificate in \Cref{thm:certificate_opt}.

We first recall the definition of a directed cycle.
\begin{definition}[Directed cycle]
Let $G=(S,E)$ be a directed labeled graph.  
A directed cycle is a sequence of edges
\[
(a_1,a_2,i_1), (a_2,a_3,i_2), \dots, (a_k,a_1,i_k) \in E,
\]
with $a_1,\dots,a_k$ distinct nodes.
\end{definition}

The following lemma shows that solving \Cref{prob:graph_opt} on a cycle yields a \emph{lower} bound on the JSR. Together with the upper bound provided by \Cref{prob:graph_opt}, this result forms the basis for the optimality certificate of \Cref{thm:certificate_opt} and the stopping criterion of Algorithm~\ref{algorithm:opti}.
\begin{lemma}\label{lemma:lb_on_cycle}
Let $G=(S,E)$ be a directed labeled graph consisting of a single directed cycle. Then $\gamma^\star (G) \le \rho(\mathcal A)$.
\end{lemma}
\ifextended
\begin{proof}
The constraints of \Cref{prob:graph_opt} read
\[
\gamma^2 P_{a_j} \succeq A_{i_j}^\top P_{a_{j+1}} A_{i_j},
\quad j=1,\dots,k,
\]
with $a_{k+1}=a_1$. Composing along the cycle gives
\[
\gamma^{2k} P_{a_1} \succeq
(A_{i_k}\cdots A_{i_1})^\top
P_{a_1}
(A_{i_k}\cdots A_{i_1}).
\]
The infimum of $\gamma^{2k}$ for which there exists
$P_{a_1}\succ0$ satisfying this inequality is known to be $\rho(A_{i_k}\cdots A_{i_1})^2$~\cite{boyd1994lmi}. Hence
\[
\gamma^\star (G)=\rho(A_{i_k}\cdots A_{i_1})^{1/k}.
\]
Since for any product $A_{i_k}\cdots A_{i_1}$, one has $\rho(\mathcal A)\ge \rho(A_{i_k}\cdots A_{i_1})^{1/k}$,
we obtain $\gamma^\star(G)\le\rho(\mathcal A)$.
\end{proof}
\fi
The following lemma shows that the value of $\gamma^\star(G)$ is determined by the
strongly connected components of $G$.
\begin{lemma}
Let $G=(S,E)$ be a directed labeled graph and
let $\mathcal S(G)\subseteq S$ denote the set of strongly connected
components of $G$.
For each $C \in \mathcal S(G)$, denote by
$G_C=(C,E_C)$ the subgraph induced by $C$.\footnote{The subgraph induced by $C$ is the graph $G_C=(C,E_C)$ where $E_C=\{(a,b,i)\in E : a,b\in C\}$.}
Then $\gamma^\star (G)
=
\max_{C \in \mathcal S(G)}
\gamma^\star (G_C)$.
\label{lemma:strongly_connected_comp}
\end{lemma}
\ifextended
\begin{proof}
The inequality
\[
\gamma^\star (G)
\ge
\max_{C\in\mathcal S(G)} \gamma^\star (G_C)
\]
is immediate since, for each $C\in\mathcal S(G)$, the constraints of~\Cref{prob:graph_opt} on $G_C$ are a subset of those for $G$. We will therefore prove the reverse inequality.

Let
\[
\bar\gamma \coloneqq
\max_{C\in\mathcal S(G)} \gamma^\star (G_C).
\]
Let $\varepsilon > 0$ and set $\gamma = \bar\gamma + \varepsilon >0$.

Consider the graph obtained by \emph{contracting} each strongly connected
component of $G$, that is, by replacing every component $C\in\mathcal S(G)$
with a single node and keeping all edges $(a,b,i)\in E$ with $a\in C$ and
$b\in C'$ as edges from $C$ to $C'$.
The resulting graph is a directed acyclic graph with $m\coloneqq |\mathcal S(G)|$ nodes.
Hence the components admit a topological ordering
$C_1,\dots,C_m$ such that edges between components are of the form
$C_i \to C_j$ with $i<j$.

We construct a feasible solution $(\gamma, \{P_a\}_{a\in S})$ for \Cref{prob:graph_opt} on $G$ by combining feasible solutions on each
strongly connected component and appropriately rescaling them to satisfy
the inter-component constraints.

For each $k$, since $\gamma>\gamma^\star(G_{C_k})$, there exists a feasible solution
$(\hat\gamma_k,\{P_a^{(k)}\}_{a\in C_k})$ to
\Cref{prob:graph_opt} on $G_{C_k}=(C_k,E_{C_k})$ such that
$\hat\gamma_k<\gamma$. Hence,
\[
\gamma^2 P_a^{(k)} \succeq
\hat\gamma_k^2 P_a^{(k)}
\succeq A_i^\top P_b^{(k)} A_i,
\quad \forall (a,b,i)\in E_{C_k}.
\]

To enforce the constraints between components, we introduce scaling
factors $\alpha_k>0$ and define them along the topological ordering.
Set $\alpha_1=1$, and suppose that $\alpha_1,\dots,\alpha_{k-1}$ have
been fixed. We choose $\alpha_k$ so that, for any edge $(a,b,i)\in E$
with $a\in C_j$, $b\in C_k$, and $j<k$,
\[
\gamma^2 \alpha_j P_a^{(j)}
\succeq
\alpha_k A_i^\top P_b^{(k)} A_i.
\]
Since $\gamma^2 \alpha_j P_a^{(j)} \succ 0$ and
$A_i^\top P_b^{(k)} A_i \succeq 0$, this is ensured by taking
$\alpha_k>0$ sufficiently small.

Define $P_a \coloneqq \alpha_k P_a^{(k)}$ for $a\in C_k$. Then $P_a\succ0$
and, by construction,
\[
\gamma^2 P_a \succeq A_i^\top P_b A_i, \quad \forall (a,b,i)\in E.
\]
Thus $\gamma$ is feasible for \Cref{prob:graph_opt} on $G$, so
$\gamma^\star(G) \le \gamma = \bar\gamma + \varepsilon$ for all
$\varepsilon>0$, hence $\gamma^\star(G) \le \bar\gamma$.
\end{proof}
\fi

The previous two lemmas yield the following result.
\begin{proposition}
\phantomsection
\label{prop:gamma_opt_when_cycle}
Let $G=(S,E)$ be a directed labeled graph. Assume that each strongly connected component of $G$ is a directed cycle or an isolated node\footnote{We call an \emph{isolated node} a strongly connected component consisting of a
single node and no self-loop.}. Then \(
\gamma^\star (G) \le \rho(\mathcal A).
\)
\end{proposition}
\ifextended
\begin{proof}
Let $\mathcal S(G)$ be the set of strongly connected components of $G$.
By \Cref{lemma:strongly_connected_comp}, there exists $C \in \mathcal S(G)$ such that
$\gamma^\star (G)=\gamma^\star (G_C)$.

If $C$ is a directed cycle, \Cref{lemma:lb_on_cycle} gives
$\gamma^\star (G_C) \le \rho(\mathcal A)$.
If $C$ is an isolated node, then $G_C$ has no edge and thus
$\gamma^\star (G_C)=0 \le \rho(\mathcal A)$, concluding the proof.
\end{proof}
\fi
\subsection{Certificate of optimality}\label{sec:opti_certificate}
Building on the previous results, we derive a structural condition on the
tight subgraph that guarantees tightness of the upper bound obtained by solving
\Cref{prob:graph_opt} on a path-complete graph. Before stating the main
theorem, we recall the following elementary result from graph theory.
\begin{lemma}\label{lemma: outgoing_edges_cycle}
Let $G=(S,E)$ be a directed labeled graph such that each node has at most one outgoing edge.
Then every strongly connected component of $G$ is either a directed cycle or an isolated node.
\end{lemma}
\ifextended
\begin{proof}
The proof is given in Appendix~\ref{appendix: proof_outgoing_edges_cycle}.
\end{proof}
\fi
 We can now state the following optimality certificate.
\begin{theorem}\label{thm:certificate_opt}
Let $G=(S,E)$ be a path-complete graph, and suppose that \Cref{prob:graph_opt} admits an optimal solution $(\gamma^\star ,\{P_a^\star \}_{a\in S})$ on $G$. Let $\tilde G=(S,\tilde E)$ be the tight subgraph associated with this solution. Assume that in $\tilde G$ every node has at
most one outgoing edge. Then $\gamma^\star (G)=\rho(\mathcal A)$.
\end{theorem}
\ifextended
\begin{proof}
By \Cref{lemma: outgoing_edges_cycle}, each strongly connected component
of $\tilde G$ is either an isolated node or a directed cycle.
Hence, by \Cref{prop:gamma_opt_when_cycle} applied on $\tilde G$,
\(
\gamma^\star (\tilde G)\le \rho(\mathcal A).
\)
Since $\gamma^\star (G)=\gamma^\star (\tilde G)$ by \Cref{prop:opt_tight}, and
$\gamma^\star (G)\ge \rho(\mathcal A)$ by \Cref{thm:ub_jsr},
the result follows.
\end{proof}
\fi
\subsection{Forward lifts and the tight subgraph}
We now study how the structure of the tight subgraph interacts with forward lifts introduced in \Cref{sec:lifting}. In particular, we identify nodes for which performing a forward lift cannot improve the bound on the JSR. The following proposition formalizes this.
\begin{proposition}
Let $G=(S,E)$ be a directed labeled graph, and suppose that \Cref{prob:graph_opt} admits an optimal solution $(\gamma^\star ,\{P_a^\star \}_{a\in S})$ on $G$. Let $\tilde G=(S,\tilde E)$ be the tight
subgraph associated with this solution. Suppose that there exists
$v\in S$ with at most one outgoing edge in $\tilde G$. Then the forward lift of $G$ with
respect to $v$, denoted $G^v$, satisfies $\gamma^\star(G^v)=\gamma^\star(G)$.\label{prop:at_most_1_tight}
\end{proposition}
\ifextended
\begin{proof}
Let $v\in S$ with at most one outgoing edge in $\tilde E$. We already know that $\gamma^\star (G^v)\le\gamma^\star (G)$ by
\Cref{thm:relaxation_forward}. Let
$(\gamma,\{\hat P_a\}_{a\in S^v})$ be feasible for
\Cref{prob:graph_opt} on $G^v$. We construct
$(\gamma,\{P_a\}_{a\in S})$ feasible for $\tilde G=(S,\tilde E)$.

If $\{(v,b,i)\in\tilde E\}=\{(v,\bar b,\bar i)\}$, set
\[
P_a\coloneqq\hat P_a \ (a\neq v),\qquad
P_v\coloneqq\hat P_{v^{(\bar b)}} .
\]
If $\{(v,b,i)\in\tilde E\}=\emptyset$, choose
$\bar b\in\mathrm{Post}(v)$ and set $P_v\coloneqq\hat P_{v^{(\bar b)}}$.

We verify feasibility for each $(a,b,i)\in\tilde E$.

If $a,b\neq v$, the edge appears identically in $G^v$, hence
\[
\gamma^2 P_a\succeq A_i^\top P_b A_i .
\]

If $a=v$ and $b\neq v$, then necessarily $b=\bar b$, and $G^v$
contains the edge $(v^{(\bar b)},\bar b,i)$, so
\[
\gamma^2 P_v=\gamma^2 \hat P_{v^{(\bar b)}}\succeq A_i^\top \hat P_{\bar b} A_i=A_i^\top P_{\bar b} A_i.
\]

If $a=b=v$, then $\bar b=v$ and $G^v$ contains
$(v^{(v)},v^{(v)},i)$, and feasibility of $G^v$ gives
\[
\gamma^2 \hat P_{v^{(v)}}\succeq A_i^\top \hat P_{v^{(v)}} A_i ,
\]
which yields $\gamma^2 P_v\succeq A_i^\top P_v A_i$.

If $a\neq v$ and $b=v$, then $(a,v^{(w)},i)\in E^v$ for any
$w\in\mathrm{Post}(v)$, hence
\[
\gamma^2 \hat P_a\succeq A_i^\top \hat P_{v^{(w)}} A_i .
\]
Taking $w=\bar b$ gives $\gamma^2 P_a\succeq A_i^\top P_v A_i$.

Thus $(\gamma,\{P_a\})$ is feasible for $\tilde G$. Therefore every feasible solution for $G^v$ yields a feasible solution for $\tilde G$ with the same objective value, implying $\gamma^\star (G^v)\ge\gamma^\star (\tilde G)$. Since
$\gamma^\star (G)=\gamma^\star (\tilde G)$ by \Cref{prop:opt_tight},
we obtain $\gamma^\star (G^v)\ge\gamma^\star (G)$, concluding the proof.
\end{proof}
\fi
\section{MAIN ALGORITHM}\label{sec:algorithm}
In this section, we combine the forward lift introduced in
\Cref{sec:lifting} with the insights on the tight subgraph from
\Cref{sec:tight} to design an iterative algorithm that refines
path-complete graphs.

In particular, performing a forward lift with respect to a node that
has at most one tight outgoing edge cannot improve the upper bound on the
JSR (see \Cref{prop:at_most_1_tight}). This significantly reduces the search space, allowing us to focus only on
nodes with at least two tight outgoing edges, which are precisely nodes for which lifting may improve the bound.

Moreover, \Cref{thm:certificate_opt} shows that if no such node exists,
then the current graph is already optimal, in the sense that the
resulting bound coincides with the JSR. In this case, the algorithm terminates with a certificate of optimality. 

The resulting procedure is summarized in \Cref{algorithm:opti}.

\begin{algorithm}[h]
\caption{Optimization-Based Graph Lifting}
\label{algorithm:opti}
\begin{algorithmic}[1]

\STATE \emph{Initialization}
\STATE Let $G_0=(S_0,E_0)$ be an arbitrary path-complete graph (e.g., the one in \Cref{fig:pc-graph} if $M=2$).
\STATE Solve \Cref{prob:graph_opt} on $G_0$ to obtain $(\gamma_0,\{P_a^{(0)}\}_{a\in S_0})$.
\STATE Let $\tilde E_0$ be the set of tight edges.

\FOR{$n = 0,1,2,\dots$}

\STATE \emph{Candidate node selection}
\STATE Define
\[
A_n \coloneqq 
\{a\in S_n : |\{(a,b,i)\in \tilde E_n\}| \ge 2 \}.
\]

\IF{$A_n = \emptyset$}
\STATE \textbf{return} $\gamma_n$ \quad (certificate of optimality)
\ENDIF
\STATE Select $v_n \in A_n$ having the largest number of tight outgoing edges. If several nodes attain this maximum, select one of them arbitrarily.
\STATE \emph{Graph lifting}
\STATE Construct the forward lift
\[
G_{n+1} \coloneqq G_n^{v_n}.
\]

\STATE \emph{Optimization step}
\STATE Solve \Cref{prob:graph_opt} on $G_{n+1}$ to obtain
$(\gamma_{n+1},\{P_a^{(n+1)}\}_{a\in S_{n+1}})$.

\STATE Let $\tilde E_{n+1}$ be the set of tight edges.

\ENDFOR

\end{algorithmic}
\end{algorithm}
\begin{remark}
\Cref{thm:certificate_opt} guarantees exactness upon termination, but provides no termination or convergence guarantee. Establishing such guarantees is left for future work.
\end{remark}
\section{NUMERICAL EXPERIMENTS}\label{sec:experiments}
In this section, we evaluate the performance of the proposed algorithm
on randomly generated systems and compare it with the De Bruijn
hierarchy. We recall the
definition of De Bruijn graphs, which form a hierarchical family of graphs indexed
by their order and provide a systematic sequence of increasingly tight
bounds on the JSR. They are therefore commonly used as a baseline when
evaluating path-complete graph constructions.

\begin{definition}[De Bruijn graph]
The \emph{(primal) De Bruijn graph} of order $\ell \in\Ne$ on $\Mset$ is the
directed graph whose nodes are the elements of $\Mset^\ell$.
There is an edge labeled $i\in\Mset$ from $(j_1,\ldots,j_\ell)$ to
$(i,j_1,\ldots,j_{\ell-1})$ for every $(j_1,\ldots,j_\ell)\in\Mset^\ell$.
\end{definition}

Our experimental setup is similar to the protocol of~\cite[Example~4.3]{debauche2026}. We generate $500$ random systems consisting of $M=6$ matrices of
dimension $n=2$. All experiments are performed in Julia using JuMP and
MOSEK on a laptop equipped with an Apple M5 processor and 16GB of RAM.

For each system, we proceed as follows. First, we run
Algorithm~\ref{algorithm:opti}, initialized with the dual\footnote{The dual
De Bruijn graph is obtained by reversing all edges of the (primal)
De Bruijn graph.} De Bruijn graph of order $1$. The algorithm is run
until two successive estimates of the JSR differ by less than $10^{-5}$,
or until no node with two tight outgoing edges\footnote{In the numerical
implementation, tight edges are identified using a prescribed numerical
tolerance to account for finite numerical precision. The theoretical
implications of this tolerance are left for future work.} remains. We denote by
$\gamma$ the resulting estimate.

Next, we solve \Cref{prob:graph_opt} on De Bruijn graphs of increasing order until the resulting estimate matches the accuracy achieved by $\gamma$.

Figure~\ref{fig:nodes_Edges} compares the size of the graphs produced by Algorithm~\ref{algorithm:opti} and by the corresponding De Bruijn graphs achieving the same JSR estimate. When a low-order De Bruijn graph suffices (yellow dots), both approaches yield graphs of comparable size. In contrast, when higher-order graphs are required, our method produces significantly smaller graphs. Across all instances, our graphs are never larger than their De Bruijn counterparts, and are strictly smaller in $46$ out of $500$ cases. A striking example requires $7776$ nodes and $46656$ edges for De Bruijn, while our method produces a graph with only $36$ nodes and $216$ edges.

Figure~\ref{fig:times} compares the computation time required to reach the same estimate. When very small graphs are sufficient, our algorithm is slightly slower than the De Bruijn hierarchy; however, the computation time remains below one second in these cases. For more challenging instances where larger De Bruijn graphs are required, our approach becomes significantly faster since it constructs the graph adaptively rather than growing the full hierarchy. In the most extreme instance of our experiments, the De Bruijn hierarchy required $2165$ seconds, while our approach terminated in $10.5$ seconds.

Overall, these results show that the proposed lifting procedure adapts
the graph structure to the system at hand, producing more
compact certificates than the De Bruijn construction. This advantage
becomes particularly pronounced when high-order graphs are required,
where the exponential growth of the De Bruijn hierarchy leads to a sharp
increase in both graph size and computation time.
\begin{figure}[H]
\centering
\begin{subfigure}{0.68\linewidth}
    \centering
    \includegraphics[width=\linewidth]{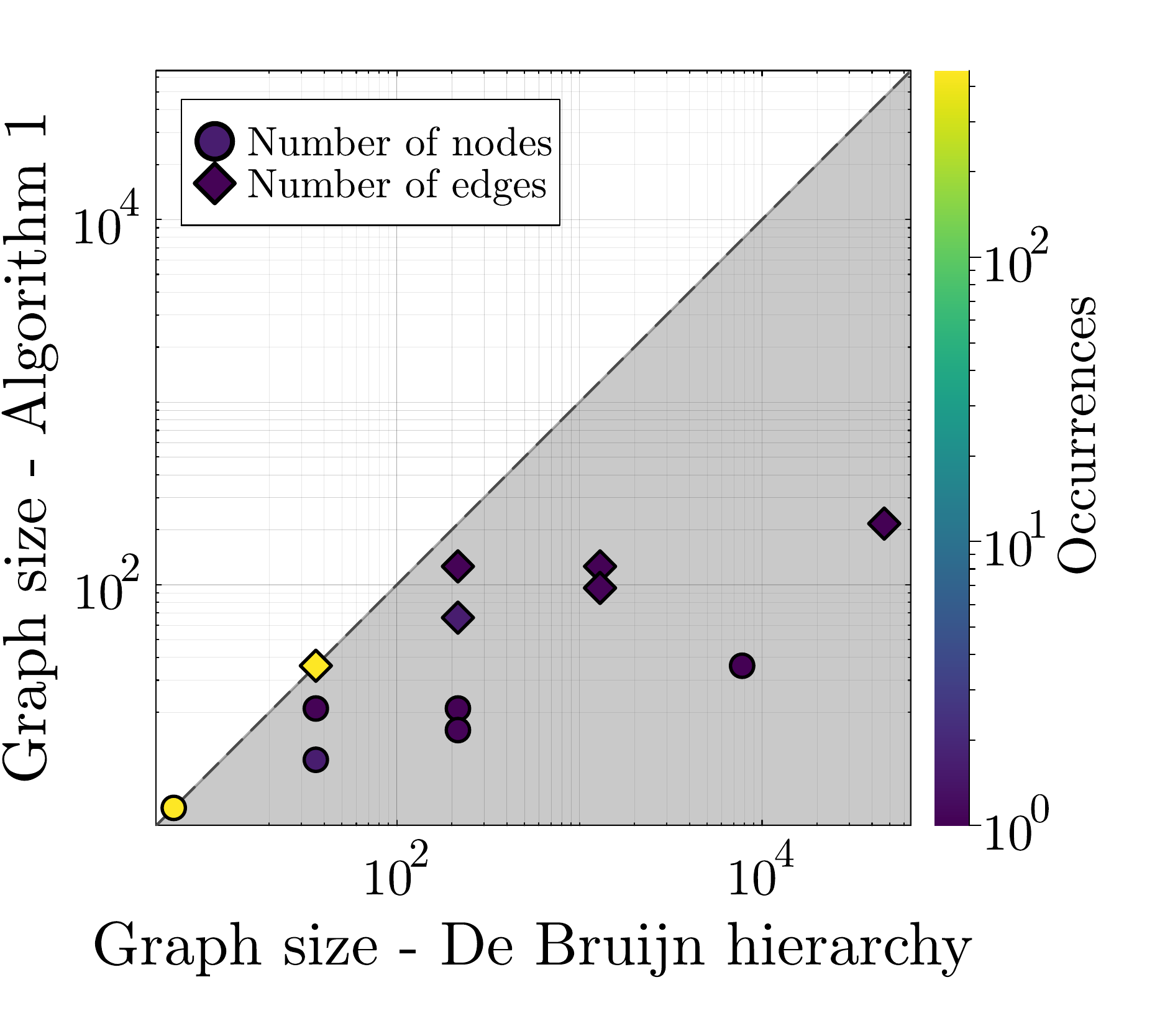}
    \caption{Number of nodes and edges}
    \label{fig:nodes_Edges}
\end{subfigure}
\hfill
\begin{subfigure}{0.68\linewidth}
    \centering
    \hspace{-0.7cm}\includegraphics[width=\linewidth]{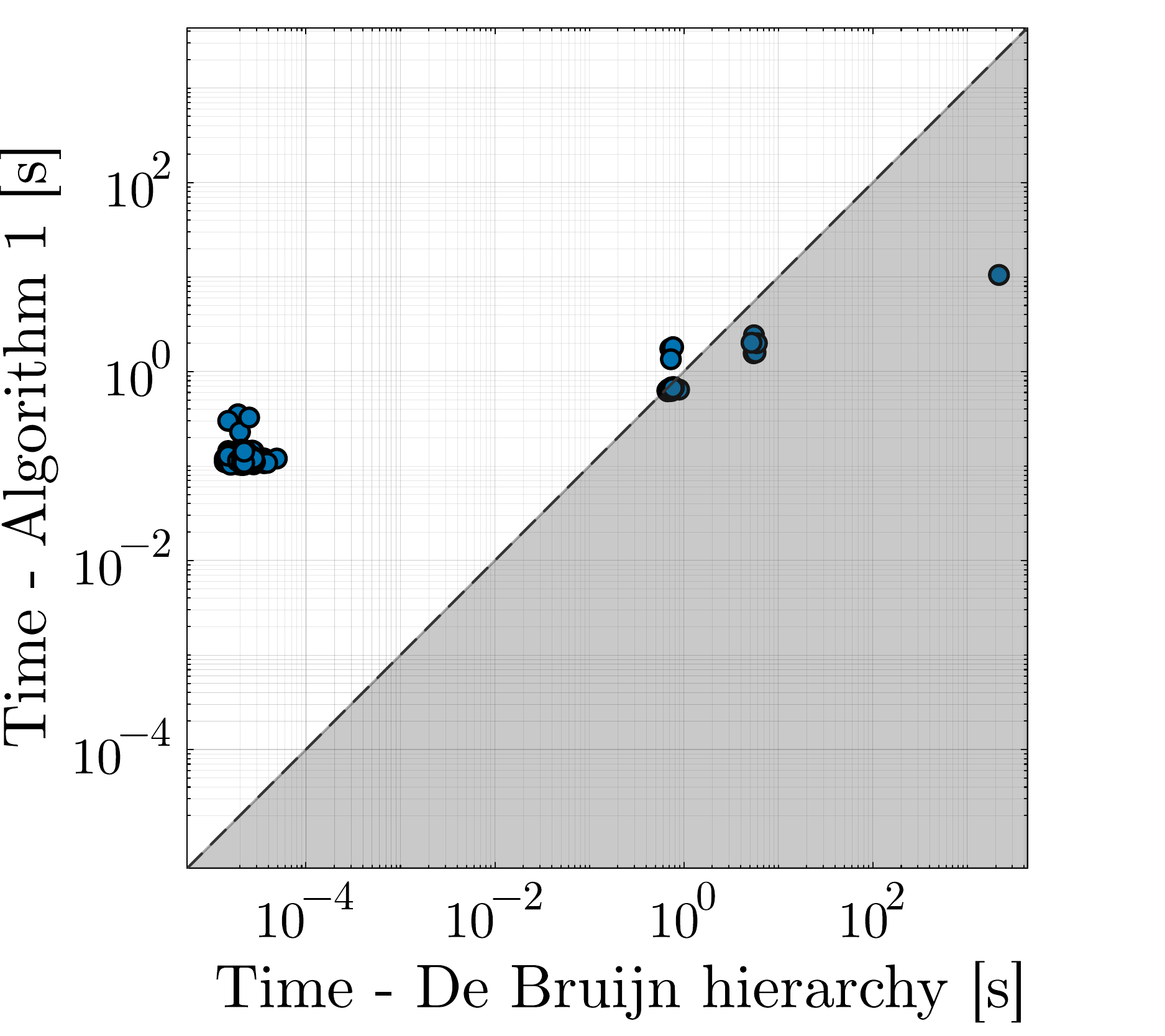}
    \caption{Computation time}
    \label{fig:times}
\end{subfigure}
\caption{Comparison between the De Bruijn hierarchy and the graph construction
from Algorithm~\ref{algorithm:opti} on $500$ random systems with $n=2$ and
$M=6$. The dashed line indicates equality. Points below the line correspond
to instances where the proposed method produces smaller graphs or requires
less computation time. Our method yields
substantially smaller graphs and faster computation times when high-order
De Bruijn graphs are required.}
\label{fig:comparison}
\end{figure}
\smallskip
To highlight the scalability limitations of the De Bruijn hierarchy with respect to the number of modes, we generate random systems and evaluate its performance for several values of $M$. \Cref{tab:comparison_M} reports cases where it fails to reach the same level of accuracy as Algorithm~\ref{algorithm:opti} within a time limit of $30$ minutes.

The table shows that, although De Bruijn graphs involve significantly
larger numbers of nodes, they fail to match the accuracy achieved by
Algorithm~\ref{algorithm:opti} within the time limit. In contrast, the
proposed method attains more accurate estimates using substantially
smaller graphs. Importantly, for all instances, the estimate returned by
Algorithm~\ref{algorithm:opti} is obtained within 40 seconds, and within
10 seconds for values of $M$ smaller than $10$.

\begin{table}[H]
\centering
\begin{tabular}{c|cc|cc}
\hline
 & \multicolumn{2}{c|}{Algorithm~\ref{algorithm:opti}} 
 & \multicolumn{2}{c}{De Bruijn} \\
$M$ & $\gamma$ & Nodes & $\gamma$ & Nodes  \\
\hline
4  & $1.82388$ & $25$ &  $1.82476$ & $4096$   \\
6 &  $2.31039$ & $26$ &  $2.31065$& $1296$  \\
8  & $2.09257$ & $29$ &  $2.092878$  & $512$ \\
10 & $2.5635$ & $73$ & $2.58689$ & $1000$  \\
\hline
\end{tabular}
\caption{Comparison between Algorithm~\ref{algorithm:opti} and the De Bruijn
hierarchy for increasing numbers of modes $M$, on challenging
instances. For each method, we report
the achieved upper bound $\gamma$ and the number of nodes in the corresponding graph. For the De Bruijn hierarchy, $\gamma$ and the number of nodes correspond
to the largest order completed within a 30-minute time limit.}
\label{tab:comparison_M}
\end{table}
Finally, the numerical results can also be put in perspective by considering
the SMT-based approach proposed in~\cite{debauche2026}, which relies on linear
copositive norms, a priori a less expressive template than the quadratic
Lyapunov functions considered here. Their approach can yield substantially more compact
graphs than the De Bruijn hierarchy, at the cost of additional computational
effort, as discussed by the authors. For instance, for one of their random
systems with \(n=2\) and \(M=5\), the De Bruijn hierarchy takes less than
\(1\) second, compared with more than \(300\) seconds for their approach,
whereas our method outperforms the De Bruijn hierarchy on the more
challenging instances.
\section{CONCLUSION}
We proposed a method for refining path-complete graphs in the
Lyapunov-based approximation of the joint spectral radius (JSR).
Our analysis shows that the structure of the tight constraints of the
optimal solution provides a simple certificate to determine whether a
given graph already yields the exact JSR bound. Based on this insight,
we designed an iterative algorithm that selects nodes to lift by
inspecting the tight subgraph and applies a local lifting operation. Numerical experiments show that achieving the same estimate as our algorithm with the De Bruijn hierarchy requires much larger graphs, resulting in significantly higher computation times. This limitation becomes more pronounced as the number of modes increases, with De Bruijn graphs failing to reach the same accuracy within the time limit, while our method remains computationally efficient.

Our results suggest that the structure of tight constraints contains
valuable information about the minimal graph structure required to
certify the JSR. Future work includes investigating alternative lifting
operations and node-selection strategies that further exploit this
structure to design optimization-guided algorithms for constructing
path-complete graphs.
\section{ACKNOWLEDGMENTS}
We thank Guillaume O.~Berger for insightful discussions.
\bibliographystyle{IEEEtran}

\bibliography{refs}

@article{blanchini2012constant,
  title={Constant and switching gains in semi-active damping of vibrating structures},
  author={Blanchini, Franco and Casagrande, Daniele and Gardonio, Paolo and Miani, Stefano},
  journal={International Journal of Control},
  volume={85},
  number={12},
  pages={1886--1897},
  year={2012},
  publisher={Taylor \& Francis}
}

@article{donkers2011stability,
  title={Stability analysis of networked control systems using a switched linear systems approach},
  author={Donkers, MCF and Heemels, WP Maurice H and Van de Wouw, Nathan and Hetel, Laurentiu},
  journal={IEEE Transactions on Automatic control},
  volume={56},
  number={9},
  pages={2101--2115},
  year={2011},
  publisher={IEEE}
}

@book{liberzon2003switching,
  title={Switching in systems and control},
  author={Liberzon, Daniel},
  volume={190},
  year={2003},
  publisher={Springer}
}

@book{jungers2009joint,
  title={The joint spectral radius: theory and applications},
  author={Jungers, Rapha{\"e}l},
  volume={385},
  year={2009},
  publisher={Springer Science \& Business Media}
}

@article{ahmadi2014joint,
  title={Joint spectral radius and path-complete graph Lyapunov functions},
  author={Ahmadi, Amir Ali and Jungers, Rapha{\"e}l M and Parrilo, Pablo A and Roozbehani, Mardavij},
  journal={SIAM Journal on Control and Optimization},
  volume={52},
  number={1},
  pages={687--717},
  year={2014},
  publisher={SIAM}
}

@article{de1948combinatorial,
  title={On a combinatorial problem},
  author={De Bruijn, Nicolaas Govert and Erd{\"o}s, Paul},
  journal={Proceedings of the Section of Sciences of the Koninklijke Nederlandse Akademie van Wetenschappen te Amsterdam},
  volume={51},
  number={10},
  pages={1277--1279},
  year={1948}
}

@INPROCEEDINGS{athanasopoulos2019polyhedral,
  author={Athanasopoulos, N. and Jungers, R. M.},
  booktitle={2019 IEEE 58th Conference on Decision and Control (CDC)}, 
  title={Polyhedral Path-Complete Lyapunov Functions}, 
  year={2019},
  volume={},
  number={},
  pages={3399-3404},
  doi={10.1109/CDC40024.2019.9029905}}

@article{blondel2000boundedness,
author = {Blondel, Vincent and Tsitsiklis, John},
year = {2000},
month = {10},
pages = {135-140},
title = {The Boundedness of All Products of a Pair of Matrices is Undecidable},
volume = {41},
journal = {Systems \& Control Letters},
}

@article{rota1960jsr,
  author  = {Rota, Gian-Carlo and Strang, Gilbert},
  title   = {A note on the joint spectral radius},
  journal = {Indagationes Mathematicae},
  volume  = {22},
  pages   = {379--381},
  year    = {1960}
}

@article{jungers2017characterization,
title = "A Characterization of Lyapunov Inequalities for Stability of Switched Systems",
author = "Jungers, \{Rapha{\"e}l M.\} and Ahmadi, \{Amir Ali\} and Parrilo, \{Pablo A.\} and Mardavij Roozbehani",
year = "2017",
volume = "62",
pages = "3062--3067",
journal = "IEEE Transactions on Automatic Control",
number = "6",
}

@article{parrilo2008approximation,
  author  = {Parrilo, Pablo A. and Jadbabaie, Ali},
  title   = {Approximation of the joint spectral radius using sum of squares},
  journal = {Linear Algebra and its Applications},
  volume  = {428},
  number  = {10},
  pages   = {2385--2402},
  year    = {2008}
}

@article{tsitsiklis1997lyapunov,
  author  = {Tsitsiklis, John N. and Blondel, Vincent D.},
  title   = {The Lyapunov exponent and joint spectral radius of pairs of matrices are hard---when not impossible---to compute and to approximate},
  journal = {Mathematics of Control, Signals, and Systems},
  volume  = {10},
  pages   = {31--40},
  year    = {1997},
  doi     = {10.1007/BF01211562}
}

@article{hernandez2011optimal,
title = {Optimal and MPC Switching Strategies for Mitigating Viral Mutation and Escape},
journal = {IFAC Proceedings Volumes},
volume = {44},
number = {1},
pages = {14857-14862},
year = {2011},
issn = {1474-6670},
author = {Esteban A. Hernandez-Vargas and Richard H. Middleton and Patrizio Colaneri}
}

@inproceedings{debauche2026,
  author    = {Alessandro Abate and Virginie Debauche and Mirco Giacobbe and Diptarko Roy},
  title     = {Succinct Synthesis of Multiple Lyapunov Certificates for Switched Systems},
  booktitle = {Proceedings of the ACM International Conference on Hybrid Systems: Computation and Control (HSCC)},
  year      = {2026}
}

@article{branicky1998multiple,
  author={Branicky, M.S.},
  journal={IEEE Transactions on Automatic Control}, 
  title={Multiple Lyapunov functions and other analysis tools for switched and hybrid systems}, 
  year={1998},
  volume={43},
  number={4},
  pages={475-482}}

@article{lee2006uniform,
title = {Uniform stabilization of discrete-time switched and Markovian jump linear systems},
journal = {Automatica},
volume = {42},
number = {2},
pages = {205-218},
year = {2006},
author = {Ji-Woong Lee and Geir E. Dullerud}
}

@article{bliman2003stability,
title = {Stability Analysis of Discrete-Time Switched Systems Through Lyapunov Functions with Nonminimal State},
journal = {IFAC Proceedings Volumes},
volume = {36},
number = {6},
pages = {325-329},
year = {2003},
author = {Pierre-Alexandre Bliman and Giancarlo Ferrari-Trecate}
}

@book{boyd1994lmi,
author = {Boyd, Stephen and El Ghaoui, Laurent and Feron, Eric and Balakrishnan, Venkataramanan},
title = {Linear Matrix Inequalities in System and Control Theory},
publisher = {Society for Industrial and Applied Mathematics},
year = {1994}
}
\ifextended
\appendix
\subsection{Proof of \Cref{prop:opt_tight}} \label{appendix:proof_equality_tight}

Let $(\gamma^\star,\{P_a^\star\})$ be an optimal solution to
\Cref{prob:graph_opt} on $G=(S,E)$. Since $\tilde E\subseteq E$, any feasible
solution of \Cref{prob:graph_opt} on $G$ is feasible on $\tilde G$, hence
\[
\gamma^\star(\tilde G)\leq\gamma^\star(G)=\gamma^\star.
\]
Assume for contradiction that
\[
\gamma^\star(\tilde G)<\gamma^\star.
\]
Then there exists a feasible solution
$(\tilde\gamma,\{\tilde P_a\})$ on $\tilde G$ such that
\[
\gamma^\star(\tilde G)\leq\tilde\gamma<\gamma^\star.
\]
For $\lambda\in[0,1]$, define
\[
\hat P_a(\lambda)=\lambda P_a^\star +(1-\lambda)\tilde P_a \succ 0.
\]
For $(a,b,i)\in E$, define
\[
F_{a,b,i}(\lambda)=
{\gamma^\star}^2 \hat P_a(\lambda)-A_i^\top\hat P_b(\lambda)A_i .
\]
We can rewrite it as:
\begin{align*}
F_{a,b,i}(\lambda)&=\lambda\left({\gamma^\star}^2 P_a^\star-A_i^\top P_b^\star A_i \right)+\left(1-\lambda\right)\left( {\gamma^\star}^2 \tilde P_a-A_i^\top\tilde P_b A_i \right)\\
&=\lambda F_{a,b,i}(1)+(1-\lambda)F_{a,b,i}(0),
\end{align*}
where $F_{a,b,i}(1)
\succeq0$ for any $(a,b,i)\in E$ by feasibility of $(\gamma^\star,\{P_a^\star\})$ on $G$.

If $(a,b,i)\in\tilde E$, then feasibility of
$(\tilde\gamma,\{\tilde P_a\})$ on $\tilde G$ and $\tilde\gamma<\gamma^\star$ give
\[
F_{a,b,i}(0)
\succeq
\bigl({\gamma^\star}^2-\tilde\gamma^2\bigr)\tilde P_a
\succ0,
\]
hence $F_{a,b,i}(\lambda)\succ0$ for every $\lambda<1$.

If $(a,b,i)\notin\tilde E$, then, by definition of the tight subgraph,
$F_{a,b,i}(1)\succ0$. By continuity, there exists
$\bar\lambda_{a,b,i}<1$ such that
\[
F_{a,b,i}(\lambda)\succ0
\qquad
\text{for all }\lambda\in(\bar\lambda_{a,b,i},1].
\]

Let
\[
\bar\lambda=
\max_{(a,b,i)\notin\tilde E}\bar\lambda_{a,b,i}
\]
and choose $\lambda\in(\bar\lambda,1)$. Then
$F_{a,b,i}(\lambda)\succ0$ for all $(a,b,i)\in E$, so
$(\gamma^\star,\{\hat P_a(\lambda)\})$ is strictly feasible for
\Cref{prob:graph_opt} on $G$. It follows that there exists
$\varepsilon>0$ such that
$(\gamma^\star-\varepsilon,\{\hat P_a(\lambda)\})$ remains feasible for
\Cref{prob:graph_opt} on $G$, contradicting optimality of $\gamma^\star$.
\subsection{Proof of \Cref{lemma: outgoing_edges_cycle}} \label{appendix: proof_outgoing_edges_cycle}
Let $C$ be a strongly connected component of $G$. If $C$ consists of a
single node with no self-loop, it is an isolated node.

Otherwise, since $C$ is strongly connected, every node in $C$ must have
an outgoing edge whose endpoint also lies in $C$. As each node has at
most one outgoing edge, each node in $C$ has exactly one successor in
$C$. Starting from any node and following successive edges therefore
produces a directed cycle contained in $C$.

If $C$ contained a node outside this cycle, it would have a path to the
cycle but no path back, contradicting strong connectivity. Hence $C$ is
a directed cycle.
\fi
\end{document}